\documentclass[leqno]{amsart}
\usepackage[normal]{boilerplate2}
\usepackage{mathrsfs}
\usepackage{tikz-cd}

\newcommand{\Mack}{\categ{Mack}}

\DeclareMathOperator{\Lin}{Lin}

\def\reflectedop#1#2{\mathop{\reflectbox{$#1{#2}$}}}

\title[Stability with respect to an orbital $\infty$-category]{Parametrized higher category theory and higher algebra: Exposé IV -- Stability with respect to an orbital $\infty$-category}

\author{Denis Nardin}
\address{Department of Mathematics, Massachusetts Institute of Technology, 77 Massachusetts Avenue, Cambridge, MA 02139-4307, USA}
\email{nardin@math.mit.edu}

\begin{document}

\begin{abstract}
In this paper we develop a theory of stability for $G$-categories (presheaf of categories on the orbit category of $G$), where $G$ is a finite group. We give a description of Mackey functors as $G$-commutative monoids exploit it to characterize $G$-spectra as the $G$-stabilization of $G$-spaces. As an application of this we provide an alternative proof of a theorem by Guillou and May. The theory here is developed in the more general setting of orbital categories.
\end{abstract}

\maketitle

\tableofcontents

\section{Introduction}

It is often said that spectra are the same as homology theories. This is strictly speaking wrong when homology theories are interpreted as valued in graded abelian groups, due to the presence of phantom maps. Luckily, Goodwillie calculus 
provides us with an equivalence between spectra and linear functors from finite pointed spaces to spaces (that is space-valued homology theories). This allows us to state a universal property for the category of spectra: it is the universal source of a linear functor to spaces (\cite[Pr.~1.4.2.22]{HA}).

One would imagine that a similar statement should be true for $G$-spectra, where $G$ is a finite group. The category of $G$-spectra is not, however, the universal source of linear functors to $G$-spaces (that would be spectral presheaves over the orbit category of $G$). It has been an important insight in the solution to the Kervaire invariant one problem by Hill, Hopkins, and Ravenel (\cite{HHRarxiv}) that in $G$-spectra one should ask for a stronger form of additivity: they should not only turn coproducts into products, but also coproducts indexed by a finite $G$-set into the corresponding product. This is merely a form of Atiyah duality for finite $G$-sets, but a highly suggestive one.

In order to speak of indexed products and coproducts it is necessary to be able to remember the notion of objects with an $H$-action for every subgroup $H$ of $G$. So we need to move from the notion of $\infty$-category to the notion of $G$-$\infty$-category, which is a presheaf of categories over $\OO_G$, the orbit category of $G$. This sends $G/H$ to the $\infty$-category of objects corresponding to the subgroup $H$ (e.g. $H$-spaces, $H$-spectra etc.) and encodes all the functoriality of restriction to subgroups (corresponding to the map $G/H\to G/K$ for $H\subseteq K$) and of twisting the action by conjugation (corresponding to the isomorphism of $G/H$ with $G/gHg^{-1}$ in $\OO_G$). The general theory of (co)limits indexed by a $G$-$\infty$-category has been developed in \cite{Exp2}. We will briefly summarize the necessary results in section \ref{sec:preliminaries}.

Once the notion of $G$-(co)limit has been set up, one can try to mimic the whole theory of additive and stable $\infty$-categories in this equivariant setting. This works nicely and provides us with a universal property for the $G$-$\infty$-category of $G$-spectra: it is the universal recipient of a $G$-linear functor from the $G$-$\infty$-category of finite $G$-spaces (cf. theorem \ref{thm:universal-property-B-spectra}).
\begin{thm}
    For any $G$-category with finite $G$-colimits $C$ the $G$-functor $\Omega^\infty:\Sp^G\to \Top^G$ induces an equivalence
    \[\Fun^{G-\mathrm{rex}}_G(C,\underline{\Sp}^G)\to \Lin_G(C,\underline{\Top}_G)\]
    between the category of $G$-functors $C\to \underline{\Sp}^G$ preserving finite $G$-colimits and the category of $G$-linear $G$-functors $C\to \underline{\Top}_G$.
\end{thm}

Another important result in the same spirit is the identification of connective spectra with group-like commutative monoids in spaces, as done in \cite{MR50:5782}. This too has an equivariant analogue (cf. corollary~\ref{cor:recognition-principle}). In fact it turns out that $G$-commutative monoids are the same thing as product-preserving functors from the effective Burnside category of \cite{M1}. This explains the ubiquity of Mackey functors in equivariant homotopy theory and allows us to give an alternative proof of \cite[Th.~0.1]{guillou2}, identifying orthogonal $G$-spectra with spectral Mackey functors (see appendix \ref{sec:orthogonal}).

Two important predecessors of this paper are \cite{MR2286026} and \cite{dotto-moi-excision}. In the first a description of $G$-spectra as enriched functors from $G$-spaces to $G$-spaces is provided for a general compact Lie group, while the second contains a characterization of $G$-spectra as functors in term of an excisivity condition for a finite group $G$. While the approach taken here is different, the intuition behind it is very similar.

In this paper we will work in the general setting of atomic orbital categories (see section \ref{sec:preliminaries}). Examples of atomic orbital categories beyond the orbit category of a (pro)finite group are an $\infty$-groupoid (thus recovering the theory of \cite{MR2271789}), the cyclonic orbit category (\cite[Df.~1.10]{cyclonic}) and  the global orbit category for finite groups (the full subcategory of $\OO_{gl}$ defined in \cite[Cn.~8.32]{schwedeglobal} spanned by completely universal finite subgroups of $\mathcal{L}$). One important non-example is the orbit $\infty$-category of a compact Lie group. This is due to the lack of a good notion of finite $G$-set stable under restriction to subgroups when $G$ is compact Lie. The reader uninterested in such generality can safely substitute $\OO_G$ every time $T$ appears in this paper.

{\bf Acknowledgments:} This paper is part of a joint project with Clark Barwick, Emanuele Dotto, Saul Glasman and Jay Shah. Many of the ideas and details of the present paper arose first during conversations with them. Other papers in this project are \cite{Exp0}, \cite{Exp1}, \cite{Exp2}, \cite{Exp3}, \cite{Exp5}, \cite{Exp6}, \cite{Exp7}, \cite{Exp8}, and \cite{Exp9}. I would like to thank Mark Behrens for help navigating the models for $G$-spectra in appendix~\ref{sec:orthogonal}. We also want to thank all the other past and present participants to the Bourbon seminar for the incredibly stimulating environment: Lukas Brantner, Peter Haine, Marc Hoyois, Akhil Mathew and Tomer Schlank.


\section{Preliminaries on equivariant (co)limits}\label{sec:preliminaries}

\begin{nul}
    We will be using extensively the theory of $T$-$\infty$-categories for a general base category $T$, developed in \cite{Exp1} and \cite{Exp2}. In this section we will recall the most important results.
    
    Motivated by the discussion of $G$-categories in the introduction, we want to study presheaves of $\infty$-categories over $T$. However, a different model for those turns out to be more convenient (e.g. allowing us to state results like theorem~\ref{thm:univpropDG}). To describe it we will make use of the following foundational result of the theory of $\infty$-categories (cfr. \cite[Th.~3.2.0.1]{HTT} and \cite[Sec.~3.3.2]{HTT}):
\end{nul}
\begin{thm}
    There is a cocartesian fibration $\mathcal{Z}\to \Cat_\infty$ such that for every $\infty$-category $S$ there is an equivalence between $\Fun(S,\Cat_\infty)$ and the $\infty$-category of cocartesian fibrations over $S$, sending $F:S\to \Cat_\infty$ to the pullback of $\mathcal{Z}\to \Cat_\infty$ along $F$.
\end{thm}
\begin{dfn}
    Motivated by the previous result, we let a \emph{$T$-$\infty$-category} to be a cocartesian fibration over $T^\op$. A \emph{$T$-functor} between two $T$-$\infty$-categories is simply a map of cocartesian fibrations (that is a map of simplicial sets over $T^\op$ that sends cocartesian arrows to cocartesian arrows). Using the simplicial nerve of \cite[Df.~1.1.5.5]{HTT} we can form the $\infty$-category of $T$-$\infty$-categories.
\end{dfn}
\begin{ntn}
    If $C$ is a $T$-$\infty$-category and $e:t\to t'$ is an edge of $T$, we denote the pushforward functor $C_{t'}\to C_t$ by $\delta_e$ or $\delta_{t/t'}$.
\end{ntn}
\begin{dfn}\label{def:finite-sets}
    For any $\infty$-category $T$, the $\infty$-category $\FF_T$ of \emph{finite $T$-sets} is the full subcategory of the category of presheaves on $T$ spanned by finite coproducts of representables. It satisfies the following universal property: for any $\infty$-category $D$ with all finite coproducts the forgetful functor
    \[\Fun^\amalg(\FF_T,D)\to \Fun(T,D)\]
    is an equivalence, where the left hand side is the category of functors preserving finite coproducts. There is a functor $\mathrm{Orbit}:\FF_T\to\FF$ (where $\FF=\FF_{\Delta^0}$ is the category of finite sets) sending every finite $T$-set to the set of summands.
\end{dfn}

For any finite $T$-set $U$ there is a $T$-$\infty$-category $\UU$, called the \emph{category of points} of $U$, which as a simplicial set over $T^\op$ is defined by
\[\UU=T^\op\times_{\FF^\op_T}\left(\FF^\op_T\right)_{U/}\,.\]
This is the left fibration classified by the functor sending $V$ to the space of arrows $[V\to U]$.

\begin{cnstr}
If $C,D$ are two $T$-$\infty$-categories, there exists a $T$-$\infty$-category $\underline{\Fun}_T(C,D)$ classified by the functor
\[b\mapsto \Fun_{T_{/b}}(C|_{T_{b/}}, D|_{T_{b/}})\]
There is an obvious evaluation $T$-functor
\[C\times_{T^\op} \underline{\Fun}_T(C,D)\to D\]
\end{cnstr}

\begin{dfn}
    For every $\infty$-category $C$ we want to construct a $T$-$\infty$-category $\underline{C}_T$ classified by the functor $b\mapsto \Fun(T^\op_{/b},C)$. This is the \emph{$T$-$\infty$-category of $T$-objects in $C$}. As a simplicial set over $T^\op$ it is given by
    \[\Mor_{/T^\op}(K,\underline{C}_T) = \Mor(K\times_{T^\op}\Fun(\Delta^1,T^\op),C)\]
    where $\Fun(\Delta^1,T^\op)$ lies above $T^\op$ with the evaluation at 0. This is a cocartesian fibration thanks to \cite[Co.~3.2.2.13]{HTT}.
    
    When $T=\OO_G$ and $C=\Top$, this is the cocartesian fibration classified by the functor sending $G/H$ to the $\infty$-category of genuine $H$-spaces (that is presheaves of spaces over $\OO_H$). One of the pleasant features of the model of $T$-$\infty$-categories we are using is that the $T$-$\infty$-category of $T$-objects has a simple universal property:
\end{dfn}

\begin{thm}\label{thm:univpropDG}
Suppose $T$ an $\infty$-category, $C$ a $T$-$\infty$-category, and $D$ an $\infty$-category. Then there is a natural equivalence
\[\Fun_{T}(C,\underline{D}_T)\simeq\Fun(C,D)\,.\]
In particular, by \cite[Th.~3.2.0.1]{HTT}, the $\infty$-category $\Fun_T(C,\underline{\Cat_\infty}_T)$ is equivalent to the $\infty$-category associated to the simplicial category of cocartesian fibrations over $C$ and under this equivalence left fibrations correspond to functors whose image lies in $\underline{\Top}_T$.
\end{thm}

\begin{dfn}
A $T$-adjunction between two $T$-$\infty$-categories $C$ and $D$ is an adjunction $F G: C\leftrightarrows D$ between the two total categories such that $F$ and $G$ are $T$-functors (that is they send cocartesian arrows to cocartesian arrows) and unit and counit lie above the identity natural transformation of the identity functor on $T$. This is the same thing as a relative adjunction in the sense of \cite[Sec.~7.3.2]{HA} such that both functors are $T$-functors. Note that the left adjoint in a relative adjunction is is automatically a $T$-functor, but this is not true for the right adjoint.
\end{dfn}

\begin{dfn}\label{dfn:T-colims}
Precomposition with the structure map $C\to T^\op$ induces a diagonal $T$-functor
\[\Delta:D\cong \underline{\Fun}_T(T,D)\to \underline{\Fun}_T(C,D)\,.\]
When this $T$-functor has a left $T$-adjoint we say that $D$ has all $C$-indexed $T$-colimits. Similarly, if it has a right $T$-adjoint we say that $D$ has all $C$-indexed $T$-limits. If a $T$-$\infty$-category $D$ has all $C$-indexed $T$-colimits (respectively $T$-limits) for every small $T$-category $C$ we say that $D$ is $T$-cocomplete (respectively $T$-complete).

A $T$-colimit indexed by a $T$-category of the form $\mathrm{pr}_2:K\times T^\op\to T^\op$ for $K$ an $\infty$-category is called a fiberwise $T$-colimit. A $T$-colimit indexed by the category of points of a finite $T$-set is called a finite $T$-coproduct. A $T$-$\infty$-category is said to be \emph{pointed} if it has both a $T$-initial and a $T$-terminal object (that are cocartesian sections of the structure map that fiberwise select the initial and the terminal object respectively) and the canonical comparison map is an equivalence.
\end{dfn}

The following proposition summarizes the results on $T$-(co)limits from \cite{Exp2} that will be needed in this paper.
\begin{prp}\label{prp:B-colims}
    Let $C$ be a $T$-$\infty$-category.
    \begin{itemize}
        \item $C$ has all $T$-colimits indexed by $K\times T^\op$ if and only if for every $b\in T$ the fiber $C_b$ has all colimits indexed by $K$ and for every edge $e:b\to b'$ in $T$ the pushforward functor $\delta_e:C_{b'}\to C_b$ preserves colimits indexed by $K$.
        \item Suppose $\FF_T$ has all fiber products (that is $T$ is orbital, cf. Df.~\ref{dfn:orbital-category}). Then $C$ has all (finite) $T$-coproducts if and only if the following two conditions are satisfied
        \begin{enumerate}
            \item for every $b\in T$ the fiber $C_b$ has all (finite) coproducts and for every edge $e:b\to b'$ the pushforward $\delta_e$ preserves (finite) coproducts;
            \item For every edge $e:b\to b'$ the pushforward $\delta_e$ has a left adjoint $\coprod_e$ satisfying the Beck-Chevalley condition: for every pair of edges $e:b\to b'$ and $e':b''\to b'$ the canonical base change natural transformation of functors from $C_{b''}$ to $C_b$
            \[\delta_e \coprod_{e'}\to \coprod_{o\in \mathrm{Orbit}(b\times_{b'}b'')} \coprod_{pr_1}\delta_{pr_2}\]
            is an equivalence, where $pr_1:o\to b$ and $pr_2:o\to b'$ are the restrictions to $o$ of the two projections from $b\times_{b'}b''$.
        \end{enumerate}
        \item $C$ has all $T$-colimits if and only if it has all fiberwise colimits and all finite $T$-coproducts.
    \end{itemize}
    Similar statements hold for $T$-limits. When $T$-products exist the right adjoint of $\delta_e$ will be denoted by $\prod_e$.
\end{prp}

\begin{dfn}
There is also a notion of $T$-Kan extension, defined exactly as for the $T$-(co)limit: if we have a $T$-functor $j:I\to J$ there is a $T$-functor induced by precomposition with $j$:
\[j^*:\underline{\Fun}_T(J,D)\to \underline{\Fun}_T(I,D)\,.\]
If $j^*$ has a left $T$-adjoint we denote it by $j_!$ and call it the \emph{left $T$-Kan extension} along $j$. Similarly, when $j^*$ has a right $T$-adjoint we call it the \emph{right $T$-Kan extension} $j_*$.
\end{dfn}

The proof of the following proposition can be found in \cite{Exp2}.
\begin{prp}
    Let $C$ be a $T$-$\infty$-category with all $T$-colimits. Then  for every map of small $T$-$\infty$-categories $j:I\to I'$ the left Kan extension along $j$ 
    \[j_!:\underline{\Fun}_T(I,C)\to \underline{\Fun}_T(I',C)\]
    exists. Similarly for $T$-limits and the right Kan extension $j_*$.
\end{prp}

\begin{ntn}
    A $T$-functor is said to be \emph{fiberwise left exact}, \emph{$T$-left exact}, \emph{fiberwise right exact}, \emph{$T$-right exact} if it preserves finite fiberwise limits, finite $T$-limits, finite fiberwise colimits and finite $T$-colimits respectively.
    We will denote the full $T$-$\infty$-subcategories of $\underline{\Fun}_T(C,D)$ preserving certain (co)limits will be denoted as in the following list:

    \vskip 3pt
    \setlength{\tabcolsep}{10pt}
    \begin{tabular}{ l  l }
        \labelitemi \,$\underline{\Fun}_T^{T-\mathrm{lex}}(C,D)$:  & finite $T$-colimits;       \\[1mm]
        \labelitemi \,$\underline{\Fun}_T^{fb-\mathrm{lex}}(C,D)$: & finite fiberwise colimits; \\[1mm]
        \labelitemi \,$\underline{\Fun}_T^\amalg(C,D)$:            & finite $T$-coproducts;     \\[1mm]
        \labelitemi \,$\underline{\Fun}_T^{T-\mathrm{rex}}(C,D)$:  & finite $T$-limits;         \\[1mm]
        \labelitemi \,$\underline{\Fun}_T^{fb-\mathrm{rex}}(C,D)$: & finite fiberwise limits;   \\[1mm]
        \labelitemi \,$\underline{\Fun}_T^\times(C,D)$:            & finite $T$-products.
    \end{tabular}

\end{ntn}


\section{Fiberwise stability}\label{sec:relative-stability}
    
    \begin{rec}
        If $C$ is an $\infty$-category with finite colimits and $D$ is an $\infty$-category with finite limits a functor $F:\fromto{C}{D}$ is called \emph{linear} if it sends the initial object of $C$ to the terminal object of $D$ and pushout squares in $C$ to pullback squares in $D$ (this functors are called \emph{pointed excisive} in \cite{HA}). In \cite[Pr.~1.4.2.13]{HA} it is proven that a pointed functor is linear if and only if the natural transformation $F\to \Omega F\Sigma$ is an equivalence. The full subcategory of $\Fun(C,D)$ spanned by linear functors is denoted $\Lin(C,D)$.
    \end{rec}
    
    \begin{dfn}
    Let $C,D$ $T$-$\infty$-categories and assume that $C$ has all finite fiberwise colimits and $D$ has all finite fiberwise limits. We say that a $T$-functor $F:\fromto{C}{D}$ is \emph{fiberwise linear} if the restriction on the fiber $F_b:\fromto{C_b}{D_b}$ is linear for every $b\in T$. We denote the full $T$-subcategory of $\underline{\Fun}_T(C,D)$ spanned by fiberwise linear functors with $\underline{\Lin}_T(C,D)$.
    \end{dfn}
    
    \begin{nul}
        First we want to show that, if $C$ is $T$-pointed, $\Lin_T(C,D)$ is a localization of the subcategory $\Fun_{T,*}(C,D)$ of functors sending the zero object to the terminal object in each fiber. To do so we introduce two additional functors $\Sigma_T:\fromto{C}{C}$ and $\Omega_T:\fromto{D}{D}$ which are the pushout (respectively pullback) of the diagrams
        \begin{equation*}
            \begin{tikzcd}
                id_C \ar{r}\ar{d} & * \\
                *\\
            \end{tikzcd}
            \textrm{ and }
            \begin{tikzcd}
                & *\ar{d}\\
                 * \ar{r}& id_D\\
            \end{tikzcd}\,.
        \end{equation*}
        Since fiberwise linearity can be checked fiberwise it is clear that a functor $F\in \Fun_{T,*}(C,D)$ is in $\Lin(C,D)$ if and only if the canonical map $\fromto{F}{\Omega_TF\Sigma_T}$ is an equivalence.
    \end{nul}
    
    \begin{lem}
        Suppose that $C$ is a pointed $T$-category. Then the $\infty$-category $\Lin_T(C,D)$ is stable.
    \end{lem}
    \begin{proof}
        It is clear that $\Lin_T(C,D)$ has finite limits and that it is pointed. If we show that $\Omega$ is an equivalence we are done by proposition 1.4.2.24 of \cite{HA}. But $\Omega$ is just postcomposition with $\Omega_T:\fromto{D}{D}$ and then it is obvious that precomposition with $\Sigma_T:\fromto{C}{C}$ is an inverse.
    \end{proof}
    
    \begin{dfn}
        We say that a $T$-$\infty$-category $D$ with all finite fiberwise limits and colimits is fiberwise stable if all fibers $D_b$ are stable.    
    \end{dfn}
    
    \begin{cnstr}
    If $C$ is a $T$-$\infty$-category we want to construct a fiberwise stabilization, that is the universal source of a fiberwise linear $T$-functor to $C$. 
    Let $\mathcal{E}(D)$ be the simplicial set over $T^\op$ such that
    \[\Mor_{/T^\op}(K,\mathcal{E}(D)) \cong \Mor_{/T^\op}(K\times \Top^{fin}_*,D)\]
    (this is an instance of the pairing construction of \cite[Cor.~3.2.2.13]{HTT}). The fiber over $b\in T^\op$ is the category $\Fun(\Top^{fin}_*,D_b)$.
    
    We let $\underline{\Sp}_T(D)$ be the simplicial subset of $\mathcal{E}(D)$ consisting of all simplices whose vertices are linear functors $\Top^{fin}_*\to D_b$. This is the same simplicial set denoted by $Stab(D)$ in \cite[Cn.~6.2.2.2]{HA}. It comes equipped with a natural map of simplicial sets $\Omega^\infty:\underline{\Sp}_T(D)\to D$ over $T^\op$ that on vertices is evaluation at $S^0$.
    \end{cnstr}
    
    \begin{prp}\label{prp:fiberwise-spectra}
        The map $\fromto{\underline{\Sp}_T(D)}{T^\op}$ is a fiberwise stable $T$-$\infty$-category.
        
        Moreover the natural functor $\Omega^{\infty}\colon\fromto{\underline{\Sp}_T(D)}{D}$ is a fiberwise left exact $T$-functor and for every pointed $T$-$\infty$-category $C$ with finite $T$-colimits the induced map
        \[\fromto{\underline{\Fun}_T^{fb-\mathrm{rex}}(C,\underline{\Sp}_T(D))}{\underline{\Lin}_T(C,D)}\]
        is an equivalence of categories.
    \end{prp}
    \begin{proof}
        From \cite[Cor.~3.2.2.13]{HTT} it follows immediately that $\mathcal{E}(D)$ is a cocartesian fibration whose cocartesian edges are those maps $(\Delta^1)^\sharp\times (\Top^{fin}_*)^\flat\to D^\natural$ that are marked. So to prove that $\underline{\Sp}_T(D)\to T^\op$ is a cocartesian fibration we need only to prove that it contains all cocartesian edges whose source is in it (that is, that $\underline{\Sp}_T(D)$ is closed under pushforward). But, by our description of cocartesian edges, the pushforward functor along an edge $e:b\to b'$ of $\mathcal{E}$ is given by
        \[(\delta_e)_*:\mathcal{E}_{b'}=\Fun(\Top^{fin}_*,D_{b'})\to \mathcal{E}_b=\Fun(\Top^{fin}_*,D_b)\,,\]
        that is postcomposition with the pushforward in $D$. Since the pushforward in $D$ preserves finite limits by definition, $(\delta_e)_*$ preserves linear functors and so $\underline{\Sp}_T(D)$ is a $T$-$\infty$-category.
        
        Note that the fiber of $\underline{\Sp}_T(D)$ over $b\in T$ is exactly the stabilization of the fiber $D_b$ and that the pushforward functors between fibers of $\underline{\Sp}_T(D)$ are the functors induced by the pushforward between the fibers of $D$. So the cocartesian fibration $\underline{\Sp}_T(D)$ has all finite fiberwise limits and colimits and is fiberwise stable. Moreover the functor $\fromto{\underline{\Sp}_T(D)}{D}$ is a $T$-functor preserving $T$-limits.
        
        Finally let us prove the universal property. Since the fibers of 
        \[\underline{\Fun}_T^{fb-\mathrm{lex}}(C,\underline{\Sp}_T(D))\textrm{ and }\underline{\Lin}_T(C,D)\]
        over $b\in T$ are 
        \[\Fun_{T_{/b}}^{fb-\mathrm{lex}}(C\times_{T^\op} T^\op_{b/},\underline{\Sp}_{T_{/b}}(D\times_{T^\op} T^\op_{b/})) \textrm{ and }\Lin_{T_{/b}}(C\times_{T^\op} T^\op_{b/},D\times_{T^\op} T^\op_{b/})\]
        respectively, up to replacing $T$ by its slice $T_{/b}$ it is enough to prove that the functor
        \[(\Omega^\infty)_*:\fromto{\Fun_T^{fb-\mathrm{lex}}(C,\underline{\Sp}_T(D))}{\Lin_T(C,D)}\]
        is an equivalence of categories (since being an equivalence can be checked on every fiber). Observe that $\Fun_T^{fb-\mathrm{lex}}(C,\underline{\Sp}_T(D))$ and $\Sp(\Lin_T(C,D))$ are the same subcategory of $\Fun_T(C\times \Top^{fin},D)$, because both are spanned by the functors $F:\fromto{C\times \Top^{fin}_T}{D}$ whose restriction to $C_b\times\Top^{fin}$ lie in $\Fun^{lex}(C_b,\Sp(D_b)) = \Sp(\Lin(C_b,D_b))$ for any $b\in T$. Then the thesis is obvious because $\Lin(C_b,D_b)$ is stable.
    \end{proof}


\section{Categories of finite $T$-sets}\label{sect:Gfin}

\begin{dfn}\label{dfn:orbital-category}
A small $\infty$-category $T$ is said to be \emph{orbital} if the category $\FF_T$ of definition \ref{def:finite-sets} has all pullbacks. An orbital category $T$ is \emph{atomic} if there are no nontrivial retracts, that is if every map with a left inverse is an equivalence.

A more in depth treatment of orbital $\infty$-categories can be found in \cite{Exp3}.
\end{dfn}

\begin{exm} The following are examples of atomic orbital categories:
    \begin{itemize}
        \item The orbit category $\OO_G$ of a (pro)finite group is atomic orbital;
        \item The category of finite sets and surjections is atomic orbital;
        \item In general any epiorbital category (\cite[Df.~2.1]{glasman-stratified}) is atomic orbital;
        \item All $\infty$-groupoids are atomic orbital categories;
        \item Every cosieve of an atomic orbital category is atomic orbital;
        \item More generally, the total category of every right fibration over an atomic orbital $\infty$-category is atomic orbital;
        \item The cyclonic orbit category of \cite[Df.~1.10]{cyclonic} is atomic orbital;
        \item The category of connected groupoids (that is groupoids of the form $BG$ for a finite group $G$) and covering maps is atomic orbital. This is the full subcategory of the global orbit category of \cite[Cn.~8.32]{schwedeglobal} spanned by the completely universal finite subgroups of $\mathcal{L}$.
    \end{itemize}
\end{exm}

\begin{nul}
    For the remainder of this paper, $T$ will be a fixed atomic orbital category. We will now construct $T$-$\infty$-categories of finite $T$-sets that will be used to parametrize the various multiplications composing the structure of a $T$-commutative monoid.
\end{nul}

\begin{dfn}\label{exm:finiteGsets} We want to construct the $T$-category classified by the functor $T\to\Cat_\infty$ sending $V$ to $\FF_{T_{/V}}$. We contemplate the arrow $\infty$-category $\Fun(\Delta^{1},\FF_{T})$ of the $\infty$-category $\FF_{T}$ of finite $T$-sets. Since $\FF_{T}$ admits all pullbacks, the target functor
\[\fromto{\Fun(\Delta^{1},\FF_{T})}{\Fun({\{1\}},\FF_{T})\cong\FF_T}\]
is a cartesian fibration. We may pull it back along the fully faithful inclusion $\into{T}{\FF_{T}}$ to obtain a cartesian fibration
\[\tau\colon\fromto{\Fun(\Delta^{1},\FF_{T})\times_{\Fun({\{1\}},\FF_{T})}T}{T}.\]
It is classified by the functor $\fromto{T^{\op}}{\Cat_{\infty}}$ that carries an orbit $V$ to the $\infty$-category $\FF_{T_{/V}}$.

We now write
\[p\colon\fromto{\underline{\FF}_T}{T^{\op}}\]
for the dual cocartesian fibration $\tau^{\vee}$ (constructed in \cite{BGN}) to the cartesian fibration $\tau$. This is now a $T$-$\infty$-category, called the $T$-$\infty$-category of finite $T$-sets, and once again it is classified by the functor $\fromto{T^{\op}}{\Cat_{\infty}}$ that carries an orbit $V$ to the nerve of the $\infty$-category $\FF_{T_{/V}}$. Its objects are arrows $I=[U\to V]$ with $U\in \FF_T$ and $V\in T$ and an arrow $[U\to V]\to [U'\to V']$ is a diagram
\begin{equation*}
\begin{tikzpicture}
\matrix(m)[matrix of math nodes,
row sep=4ex, column sep=4ex,
text height=1.5ex, text depth=0.25ex]
{U& W & U'\\
 V & V' & V'\\};
\path[>=stealth,->,font=\scriptsize]
(m-1-1) edge (m-2-1)
(m-1-2) edge (m-1-1)
        edge (m-1-3)
        edge (m-2-2)
(m-2-2) edge (m-2-1)
        edge[-,double distance=1.5pt] (m-2-3)
(m-1-3) edge (m-2-3);
\end{tikzpicture}
\end{equation*}
where the left square is cartesian. Composition is then defined by forming suitable pullbacks. The target functor
\begin{equation*}
\goesto{[\fromto{U}{V}]}{V}
\end{equation*}
is the structure map $p\colon\fromto{\underline{\FF}_T}{T^\op}$.
\end{dfn}

\begin{exm}
If $T=\OO_G$ is the orbit category of a profinite group $G$, then $\underline{\FF}_G\to \OO_G^\op$ is the cocartesian fibration classified by the functor sending an orbit $G/H$ to the category of finite $H$-sets (under the canonical identification that sends a finite $G$-set over $G/H$ to the fiber over $eH$).
\end{exm}

\begin{exm}\label{exm:IV} Suppose $V$ is an object of $T$. Then there is an object
\[I(V)=[\id\colon\fromto{V}{V}]\]
of $\underline{\FF}_T$, which enjoys the following property. For any object $J=[g\colon\fromto{X}{Y}]$ of $\underline{\FF}_T$, one has an equivalence
\[\Map_{\underline{\FF}_T}(J,I(V))\simeq\Map_{T}(V,Y).\]
In particular, the assignment $\goesto{V}{I(V)}$ defines a fully faithful right adjoint to the structure map $p\colon\fromto{\underline{\FF}_T}{T^{\op}}$.
\end{exm}

\begin{nul}
    In what follows it will be convenient to have at our disposal more general categories whose objects are finite $T$-sets. They will be all more easily manipulated as subcategories of the Burnside category of finite $T$-sets. The latter is the dual version of the main construction in \cite{Aefffib}, but we will repeat it here both because of its simplicity and because we will need to use some details in our main results.
\end{nul}

\begin{cnstr}
Let again us consider the cartesian fibration
\[\tau\colon\fromto{S_T:=\Fun(\Delta^{1},\FF_{T})\times_{\Fun({\{1\}},\FF_{T})}T}{T}.\]
It is also, for much easier reasons, a cocartesian fibration.

With this in mind, we now proceed to define triple structures on these $\infty$-categories. Denote by $\iota T\subset T$ the subcategory consisting of the equivalences of $T$. Then we can contemplate the triple structures
\[(T,\iota T,T)\textrm{\quad and\quad}(S_T,S_T\times_T\iota T,S_T).\]
It is a simple matter to see that these triple structures are adequate in the sense of \cite[Df.~5.2]{M1}. We may therefore construct their effective Burnside $\infty$-categories, and the projection induces a functor
\[t'\colon\fromto{A^{\eff}(S_T,S_T\times_T\iota T,S_T)}{A^{\eff}(T,\iota T,T)}.\]

An object of $A^{\eff}(S_T,S_T\times_T\iota T,S_T)$ is a morphism $[\fromto{U}{V}]$ of finite $T$-sets in which $V\in T$. If
\[I=[\fromto{U}{V}]\textrm{\quad and\quad}J=[\fromto{X}{Y}]\]
are two objects, then a morphism $\fromto{I}{J}$ of $A^{\eff}(S_T,S_T\times_T\iota T,S_T)$ is a commutative diagram
\begin{equation*}
\begin{tikzpicture}
\matrix(m)[matrix of math nodes,
row sep=4ex, column sep=4ex,
text height=1.5ex, text depth=0.25ex]
{U&W&X\\
V&Z&Y\\};
\path[>=stealth,->,font=\scriptsize,inner sep=0.5pt]
(m-1-1) edge (m-2-1)
(m-1-2) edge (m-1-1)
edge (m-2-2)
edge (m-1-3)
(m-1-3) edge (m-2-3)
(m-2-2) edge (m-2-1)
edge node[below]{$\sim$} (m-2-3);
\end{tikzpicture}
\end{equation*}
in which the morphism $\equivto{Z}{Y}$ is an equivalence in $T$.
\end{cnstr}

\begin{lem} The functor $t'$ above is both a cartesian and a cocartesian fibration. Furthermore, any morphism of $A^{\eff}(S_T,S_T\times_T\iota T,S_T)$ represented as a commutative diagram
\begin{equation*}
\begin{tikzpicture}
\matrix(m)[matrix of math nodes,
row sep=4ex, column sep=4ex,
text height=1.5ex, text depth=0.25ex]
{U&W&X\\
V&Z&Y\\};
\path[>=stealth,->,font=\scriptsize,inner sep=0.5pt]
(m-1-1) edge  (m-2-1)
(m-1-2) edge (m-1-1)
        edge (m-2-2)
        edge (m-1-3)
(m-1-3) edge (m-2-3)
(m-2-2) edge (m-2-1)
        edge node[below]{$\sim$} (m-2-3);
\end{tikzpicture}
\end{equation*}
is
\begin{itemize}
    \item $t'$-cartesian if the morphisms $\equivto{W}{U}$ and $\equivto{W}{X}$ are equivalences;
    \item $t'$-cocartesian if the left square is cartesian and $\equiv{W}{X}$ is an equivalence.
\end{itemize}
\begin{proof} This follows immediately from (the opposite of) the ``omnibus theorem'' for effective Burnside $\infty$-categories \cite[Th. 12.2]{M1}.
\end{proof}
\end{lem}

\begin{dfn} We have an inclusion $\into{T^{\op}}{A^{\eff}(T,\iota T,T)}$, which is a weak equivalence. Write
\[\underline{\AA}^{\eff}(T)\coloneq A^{\eff}(S_T,S_T\times_T\iota T,S_T)\times_{A^{\eff}(T,\iota T,T)}T^{\op};\]
the projection $\fromto{\underline{\AA}^{\eff}(T)}{A^{\eff}(S_T,S_T\times_T\iota T,S_T)}$ is thus an equivalence, and the projection
\[t\colon\fromto{\underline{\AA}^{\eff}(T)}{T^{\op}}\]
is a cartesian and cocartesian fibration, so it is a $T$-$\infty$-category.

It is classified by the functor sending $V\in T$ to the Burnside category $A^{\eff}(\FF_{T_{/V}})$.
\end{dfn}

\begin{nul}
    Note that $\underline{\FF}_T$ is naturally a $T$-subcategory of $\underline{\AA}^{\eff}(T)$, consisting of all objects and all morphisms such that the left square is cartesian. This is the analogue of the classical inclusion of the category $\FF$ of finite sets inside the Burnside category of finite sets $A^{\eff}(\FF)$ by considering only the egressive maps.
    
    This can be extended to an inclusion of \emph{pointed} finite sets $\FF_*$ inside $A^{\eff}(\FF)$ as the subcategory containing all objects and as maps the spans $[I\ot \tilde I \to I']$ such that the ``left leg'' is an inclusion ($\tilde I$ under this identification corresponds to the preimage of $I'$ under the map $I_+\to {I'}_+$, so that we can identify $\FF_*$ with the category of finite sets and partially defined maps. It will be convenient for us to turn this into the definition of finite pointed $T$-sets.
\end{nul}

\begin{dfn}\label{exm:pointedGsets}
We'll say that a map $U\to U'$ of finite $T$-sets is a \emph{summand inclusion} if there is $U''\to U'$ such that the map $U\amalg U''\to U'$ is an equivalence.

Consider the subcategory of $\underline{\AA}^{\eff}(T)$ containing all objects and whose morphisms are those diagrams
\begin{equation*}
\begin{tikzpicture}
\matrix(m)[matrix of math nodes,
row sep=4ex, column sep=4ex,
text height=1.5ex, text depth=0.25ex]
{U& \tilde U & U'\\
 V & V' & V'\\};
\path[>=stealth,->,font=\scriptsize]
(m-1-1) edge (m-2-1)
(m-1-2) edge (m-1-1)
        edge (m-1-3)
        edge (m-2-2)
(m-2-2) edge (m-2-1)
        edge[-,double distance=1.5pt] (m-2-3)
(m-1-3) edge (m-2-3);
\end{tikzpicture}
\end{equation*}
such that the arrow $\tilde U\to U\times_V V'$ is a summand inclusion (this is a condition of the left square of the diagram and does not depend on the particular choice of pullback $U\times_V V'$). This subcategory contains all cocartesian morphisms of $\underline{\AA}^{\eff}(T)\to T^\op$ and so it is a $T$-subcategory. We will call it the \emph{$T$-$\infty$-category of finite pointed $T$-sets} and denote it by $\underline{\FF_*}_T$.
\end{dfn}

\begin{ntn}\label{ntn:subscript_+}
    We will often decorate an object $I=[U\to V]$ of $\underline{\FF_*}_T$ with a subscript $+$, to remind ourselves that we see it as living in $\underline{\FF_*}_T$ rather than of $\underline{\FF}_T$ or $\underline{\AA}(T)$. The $+$ does not have any real meaning (in our construction there are no ``basepoints'') and it is only a mnemonic aid. The canonical inclusion $\fromto{\FF^T}{\FF^T_*}$ will be indicated by $(-)_+\colon I\mapsto I_+$.
\end{ntn}

\begin{lem}
The cocartesian fibration $\fromto{\underline{\FF_*}_T}{T^\op}$ is classified by the functor sending $V$ to the category of pointed objects in $(\FF_T)_{V/}$.

Moreover the canonical inclusion $(-)_+\colon\fromto{\underline{\FF}_T}{\underline{\FF_*}_T}$ has a right $T$-adjoint sending $[U\to V]$ to $[U\amalg V\to V]$.
\end{lem}
\begin{proof}
    The fiber of $\underline{\FF_*}_T$ over $V$ consists in the Burnside category of $\FF_{T_{/V}}$ where the egressive morphisms are the summand-inclusions. We can identify this with the category of pointed objects by sending a span
    \[U \ot \tilde U \to U'\]
    where $U = \tilde U \amalg W$ to the map
    \[U\amalg V \to U'\amalg V\]
    where the central map is $U\amalg V = \tilde U \amalg (W\amalg V)\to U'\amalg V$, given by $\tilde U\to U'$ on the first component and the structure map to $V$ on the second component.
    
    Now it is clear that the functor
    \[[U\amalg V\to V]\mapsto [U\amalg V\to V]\]
    is the right adjoint to the inclusion of $\FF_{T_{/V}}$ into pointed objects. So in order to have a $T$-adjunction we need only to verify that the right adjoint provided by \cite[Pr.~7.3.2.6]{HA} is a $T$-functor, but this follows from the universality of finite coproducts in $\FF_T$ and the fact that coproducts therein are disjoint.
\end{proof}

\section{$T$-semiadditive functors and $T$-semiadditive categories}

\begin{ntn}\label{ntn:characteristic_map}
    Let $I=[U\to V]\in\underline{\FF}_T$ and $W\in \mathrm{Orbit}(U)$, then the canonical map $W\to U\times_V W$ must be a summand-inclusion, since it factors through an unique orbit, of which $W$ is a retract. So we can define the \emph{characteristic map}
    \[\chi_{[W\subseteq U]}\colon I_+\to I(W)_+\]
    (where $I(W)$ is the construction of example \ref{exm:IV}) as the map of pointed finite $T$-sets described by the following diagram
    \begin{equation*}
\begin{tikzpicture}
\matrix(m)[matrix of math nodes,
row sep=4ex, column sep=4ex,
text height=1.5ex, text depth=0.25ex]
{U& W & W\\
 V & W & W\\};
\path[>=stealth,->,font=\scriptsize]
(m-1-1) edge (m-2-1)
(m-1-2) edge (m-1-1)
        edge[-,double distance=1.5pt] (m-1-3)
        edge[-,double distance=1.5pt] (m-2-2)
(m-2-2) edge (m-2-1)
        edge[-,double distance=1.5pt] (m-2-3)
(m-1-3) edge[-,double distance=1.5pt] (m-2-3);
\end{tikzpicture}\,.
\end{equation*}
Note that this map is in $\underline{\FF_*}_T$ due to the fact that, thanks to the atomicity of $T$, the map $W\to U\times_V W$ is a summand inclusion, since the orbit it factors through retract onto $W$..
\end{ntn}

\begin{cnstr}
Let $C$ be a pointed $T$-$\infty$-category with all finite $T$-coproducts, $I=[U\to V]\in\underline{\FF}_T$ and $X\in\Fun_T(\UU,C)$ be a diagram. Then there is a map induced on the colimits
\[(\chi_{[W\subseteq U]})_*:\delta_{W/V}\coprod_I X\to X_{[W\subseteq U]}\,.\]
above $\chi_{[W\subseteq U]}$, where $\coprod_I$ is the left adjoint to $\delta_I:C_V\cong\Fun_T(\VV,C)\to \Fun_T(\UU,C)$ and $X_{[W\subseteq U]}$ is the value of $X$ at $[W\subseteq U]\in\UU$. We can describe it as follows: by the base change condition in proposition \ref{prp:B-colims} we have
\[\delta_{W/V} \coprod_I X \cong \coprod_{[U\times_V W/W]}\delta_{[U\times_V W/U]} X\]
As before, the atomicity of $T$ implies that we can write
\[U\times_V W\cong W \amalg \tilde U\]
where $W$ on the right hand side is the diagonal copy. Hence
\[\delta_{W/V} \coprod_I X \cong X_{[W\to U]} \amalg \coprod_{[\tilde U\to U]} X\,.\]
So we can define a map
\[(\chi_{[W\subseteq U]})_*:\delta_{W/V}\coprod_I X \to X_{[W\subseteq U]}\]
which is the identity on the first summand and the zero map on the other.

If $D$ is a $T$-$\infty$-category with finite $T$-products and $F:C\to D$ a $T$-functor we will denote by $(\chi_{[W\subseteq U]})_*$ also the natural transformation
\[F\left(\coprod_I X\right) \to \prod_{W/V}F(X)\]
obtained by adjunction on $F(\chi_{[W\subseteq U]})_*$, where $\prod_{W/V}$ is the right adjoint of $\delta_{W/V}$.
\end{cnstr}

\begin{dfn}\label{dfn:additive-functor}
    Let $C$ be a pointed $T$-$\infty$-category with all finite $T$-coproducts and $D$ a $T$-$\infty$-category with all finite $T$-products. Then a $T$-functor $F:C\to D$ is said to be \emph{$T$-semiadditive} if for every $I=[U\to V]\in\underline{\FF}_T$ and $X\in\Fun_T(\UU,C)$ the map
    \begin{equation}\label{eqn:additive-functor}
        \prod_{W\in\mathrm{Orbit}(U)}(\chi_{[W\subseteq U]})_*:F\left(\coprod_I X\right)\to \prod_{W\in\mathrm{Orbit}(U)} \prod_{W/V} F(X_{[W\subseteq U]})\cong \prod_I F(X)
    \end{equation}
    is an equivalence. We will denote the $T$-$\infty$-category of all $T$-semiadditive $T$-functors with $\underline{\Fun}^\oplus_T(C,D)$.
    
    We say that a pointed $T$-$\infty$-category with all finite $T$-products and $T$-coproducts is \emph{$T$-semiadditive} if the identity functor is $T$-semiadditive. That is, if the map
    \begin{equation}\label{eqn:additive-category}
        \coprod_I X\to \prod_I X
    \end{equation}
    is an equivalence for every $I$-uple $X$.
\end{dfn}

It is clear that if $F:C\to D$ preserves finite $T$-coproducts and $G:D\to E$ is $T$-semiadditive then the composition $GF$ is $T$-semiadditive. Similarly if $F$ is $T$-semiadditive and $G$ preserves finite $T$-products.

\begin{exm}
    Let $C$ be a pointed $T$-$\infty$-category with all finite $T$-coproducts and $D$ an $\infty$-category with all finite products. Then the category of $T$-objects $\underline{D}_T$ has all finite $T$-products and a $T$-functor $C\to \underline{D}_T$ is $T$-semiadditive if and only if the associated functor $F:C\to D$ is such that for every $[U\to V]\in\underline{\FF}_T$ the map
    \[F\left(\coprod_I X\right)\to \prod_{W\in\textrm{Orbit}(U)} F(X_{[W\subseteq U]})\]
    is an equivalence. In particular if $D$ is semiadditive (i.e. it has biproducts), then $\underline{D}_T$ is $T$-semiadditive.
\end{exm}
\begin{exm}\label{exm:Burnside-semiadditive}
    The $T$-$\infty$-category $\AA^{\eff}(T)$ is $T$-semiadditive. In fact every fiber is semiadditive by proposition 4.3 of \cite{M1}, so it is sufficient to observe that for any arrow $W\to V$ in $T$ the functor
    \[\delta_{W/V}\colon A^{\eff}(T_{/V})\to A^{\eff}(T_{/W})\]
    has well behaved left and right adjoints and the canonical comparison map is an equivalence.
\end{exm}
\begin{cnstr}\label{cnstr:semiadditive-with-products}
    Let $C$ be a pointed $T$-$\infty$-category with finite $T$-products. Then if $I=[U\to V]\in \FF^T$ and $X:\UU\to C$, let us consider for every $W\in\mathrm{Orbit}(U)$ the map of $C_V$
    \[\eta_{[W\subseteq U]}:\into{\coprod_{W/V}X_{[W\to U]}}{\coprod_I X} \to \prod_I X\,.\]
    This can be described as the adjoint to the map in the fiber over $W$
    \[X_{[W\subseteq U]}\to \delta_{W/V}\prod_{I_W} X\cong \prod_{W'\in\mathrm{Orbit}(U\times_VW)} \prod_{W'/W} X_{[W'\to U]}\,,\]
    given by the identity map $X_W\to \prod_{W'/W}X_{W'}$ when $W'$ is the diagonal copy $W$ in $U\times_VW$ and the zero map on the other components.
    
    Then $C$ being $T$-semiadditive is equivalent to the fact that $\{\eta_W\}_{W\in\mathrm{Orbit}(U)}$ assemble to an equivalence 
    \[\coprod_I X \cong \coprod_{W\in\mathrm{Orbit}(U)} \coprod_{W/V} X_{[W\subseteq U]} \xrightarrow{\coprod \eta_{[W\subseteq U]}} \prod_I X\,.\]
\end{cnstr}

The previous remark immediately yields the following criterion for determining when a category is $T$-semiadditive
\begin{lem}
    Let $C$ be a pointed $T$-$\infty$-category with finite products and suppose that for every $I=[U\to V]\in \FF^T$ there is a natural transformation
    \[\mu^I:\prod_I \Delta X \to X\]
    of functors $C_V\to C_V$, where $\Delta:C_V\to \Fun_T(\UU,C)$ is the functor of definition \ref{dfn:T-colims}, such that
    for every $W\in\mathrm{Orbit}(U)$ the composition
    \[\mu^I \circ \eta_{[W\subseteq U]}:\coprod_{W/V} \delta_{W/V}X\to X\]
    is homotopic to the counit of the adjunction $\coprod_{W/V} \dashv\delta_{W/V}$.
    Then $C$ is $T$-semiadditive.
\end{lem}
\begin{proof}
    We need to prove that for every $[U\to V]\in\underline{\FF}_T$, $X\in \Fun_T(\UU,C)$ and $Y\in C_V$, the map
    \[\prod_{W\in\mathrm{Orbit}(U)} (\eta_{[W\subseteq U]})^*:\Map_{C_V}\left(\prod_I X,Y\right)\to \prod_{W\in\mathrm{Orbit}(U)} \Map_{C_V}\left(\coprod_{W/V}X_{[W\subseteq U]},Y\right)\]
    is an equivalence. But using the $\mu^I$ we can construct an inverse
    \[\begin{split}(\mu^I_Y)_*\circ \prod_I:\prod_{W\in\mathrm{Orbit}(U)} \Map_{C_V}\left(\coprod_{W/V}X_{[W\subseteq U]},Y\right)\cong \Map_{\Fun(\UU,C)}\left(X,\Delta Y\right)\to\\
    \to \Map_{C_V}\left(\prod_I X,\prod_I \Delta I\right)\to \Map_{C_V}\left(\prod_I X,Y\right)\,.\end{split}\]
\end{proof}

\begin{prp}
    If $C$ is a pointed $T$-$\infty$-category with finite $T$-coproducts and $D$ is a $T$-$\infty$-category with finite $T$-products then $\underline{\Fun}^\oplus(C,D)$ is $T$-semiadditive
\end{prp}
\begin{proof}
    First let us note that by $I$ to be empty in \ref{eqn:additive-functor} every $T$-semiadditive functor must send the zero object of $C$ to the terminal object of $D$. Then for any additive functor $F$ the left $T$-Kan extension of the restriction to the zero object of $C$ is the constant functor at the terminal object, since the $T$-colimit of a constant functor at the zero object is the zero object. So, if $i:\{0\}\subseteq C$ is the inclusion of the zero object
    \[\Map_{\Fun_T^\oplus}(*,F) = \Map(i_!(F|_0),F) = \Map(F|_0,F|_0)=*\] 
    Hence the constant functor at the terminal object is the zero object of $\underline{\Fun}^\oplus_T(C,D)$.
    
    Then we need to prove that $\underline{\Fun}_T^\oplus(C,D)$ satisfies the hypothesis of the previous lemma. But this is easy: for $F$ a $T$-semiadditive $T$-functor remember that $\left(\prod_I F\right)(-)\cong F\left(\coprod_I-\right)$ so we can choose
    \[\mu^I:\left(\prod_I F\right)(-) \cong F\left(\coprod_I-\right)\to F(-)\]
    given by precomposition with the canonical map $id_C\to \coprod_I$ provided by the universal property of the coproduct of $C$. Since the required identites are easily verified we are done.
\end{proof}

\begin{dfn}
    Let $C$ be a $T$-$\infty$-category with finite products. Then a $T$-commutative monoid is a $T$-semiadditive functor $\underline{\FF_*}_T\to C$. We will indicate the $T$-$\infty$-category of $T$-commutative monoids in $C$ with $\underline{\CMon}_T(C)$. Precomposition with the cocartesian section $I(-)_+:T^\op\to \underline{\FF_*}_T$ induces a $T$-functor
    \[\underline{\CMon}_T(C)\to C\,.\]
\end{dfn}

In order to prove the universal property of $\underline{\CMon}_T(C)$ we will need the following lemma
\begin{lem}\label{lem:free-pointed-category-with-coproducts}
    Let $C$ be a pointed $T$-$\infty$-category with finite $T$-coproducts. Then the map
    \[\Fun_T^\amalg(\underline{\FF_*}_T,C)\to C\]
    given by precomposition with $I(-)_+$ is an equivalence
\end{lem}
\begin{proof}
    We can construct an inverse by sending every $c\in C_V$ to the left Kan extension of its cocartesian section $\VV\to C$ along $I(-)_+:T^\op\to \underline{\FF_*}_T$.
\end{proof}

\begin{prp}
    Let $C$ be a $T$-$\infty$-category with finite $T$-products. The functor
    \[\underline{\CMon}_T(C)\to C\,.\]
    induced by precomposition with the cocartesian section $I(-)_+:T\to \FF^T_*$ is an equivalence if and only if $C$ is $T$-semiadditive.
\end{prp}
\begin{proof}
    If the map is an equivalence then $C$ is $T$-semiadditive, since $\underline{\CMon}_T(C)$ is. Vice versa if $C$ is $T$-semiadditive then
    \[\underline{\CMon}_T(C)=\underline{\Fun}^\oplus_T(\underline{\FF_*}_T,C)=\underline{\Fun}^\amalg_T(\underline{\FF_*}_T,C)\to C\]
    is an equivalence by lemma \ref{lem:free-pointed-category-with-coproducts}.
\end{proof}

\begin{cor}\label{cor:universal-property-monoids} 
    Let $C$ be a $T$-$\infty$-category with finite $T$-products and $D$ a pointed $T$-$\infty$-category with finite $T$-coproducts. Then the map
    \[\Fun_T^\amalg(D,\underline{\CMon}_T(C))\to \Fun_T^\oplus(D,C)\]
    is an equivalence of categories.
\end{cor}
\begin{proof}
    Observe that
    \[\Fun_T^\amalg(D,\underline{\CMon}_T(C)) \cong \Fun_T^\oplus(D,\underline{\CMon}_T(C))\cong\underline{\CMon}_T(\Fun_T^\oplus(D,C))\]
    where the first equivalence comes from the $T$-semiadditivity of $\underline{\CMon}_T(C)$. Since $\Fun_T^\oplus(D,C)$ is $T$-semiadditive the thesis follows by the previous proposition.
\end{proof}


\section{$T$-commutative monoids and Mackey functors}

\begin{nul}
The notion of $T$-commutative monoid, while being the natural generalization of $\Gamma$-space to the parametrized setting, might seem abstract and difficult to work with. The aim of this section is that in fact $T$-commutative monoids are just objects very familiar in equivariant homotopy theory: Mackey functors. 
\end{nul}

\begin{lem}
    Let $\underline{\FF_*^{in}}_T$ be the $T$-subcategory of $\underline{\FF_*}_T$ containing all objects and all the maps represented by spans in $S_T$
    \[I\ot \tilde I \to I'\]
    where the right arrow is an equivalence. Then a functor $M:\underline{\FF_*}_T\to \Top$ is a $T$-commutative monoid if and only if its restriction to $\underline{\FF_*^{in}}_T$ is a right Kan extension along $I(-)_+\colon T^\op\to \FF^T_*$.
\end{lem}
\begin{proof}
    Obvious from the limit description of right Kan extensions (see \cite[Pr.~4.3.2.15]{HTT}).
\end{proof}

\begin{lem}
    The inclusion $j:\underline{\FF_*}_T\to \underline{\AA}^{\eff}(T)$ is a $T$-commutative monoid.
\end{lem}
\begin{proof}
    Clear from the fact that $\underline{\FF_*}_T$ contains all $T$-coproduct diagrams of $\underline{\AA}^{\eff}(T)$ and example \ref{exm:Burnside-semiadditive}.
\end{proof}

\begin{lem}\label{lem:Kan-extension-of-monoids}
    Let $C$ be an $\infty$-category and let $E,F,D$ subcategories of $C$ such that $(C,E,D)$ and $(C,E,F)$ are adequate triples in the sense of \cite{M1}. Consider the diagram of categories
    \begin{equation*}
        \begin{tikzcd}
            A^{\eff}(C,\iota E,D) \ar{r}{f}\ar{d}{g'} & A^{\eff}(C,E,D)\ar{d}{g}\\
            A^{\eff}(C,\iota E,F) \ar{r}{f'} & A^{\eff}(C,E,F)\\
        \end{tikzcd}
    \end{equation*}
    Then if a right Kan extension along $g'$ exists so does the right Kan extension along $g'$ and the natural transformation
    \[{f'}^* g_* \to g'_* f^*\]
    is an equivalence.
\end{lem}
\begin{proof}
    Let us fix $I\in C$. We need to prove that the functor
    \begin{equation*}
        A^{\eff}(C,\iota E,D)\times_{A^{\eff}(C,\iota E,F)} A^{\eff}(C,\iota E,F)_{I/} \to A^{\eff}(C,E,D)\times_{A^{\eff}(C,E,F)} A^{\eff}(C,E,F)_{I/}
    \end{equation*}
    is coinitial. This is equivalent to the fact that for every $J\in C$ with a map from $I$ the category
    \begin{multline*}
        \left(A^{\eff}(C,\iota E,D)\times_{A^{\eff}(C,\iota E,F)} A^{\eff}(C,\iota E,F)_{I/}\right)\times_{\left(A^{\eff}(C,E,D)\times_{A^{\eff}(C,E,F)} A^{\eff}(C,E,F)_{I/}\right)}\\
        \left(A^{\eff}(C,E,D)\times_{A^{\eff}(C,E,F)} A^{\eff}(C,E,F)_{I/}\right)_{/J}
    \end{multline*}
    is weakly contractible. Let us start naming names. We have a fixed map $I\to J$ in $A^{\eff}(C,E,F)$. This correspond to a span
    \begin{equation*}
        I \xleftarrow{F} \tilde J \xrightarrow{E} J\,,
    \end{equation*}
    where we decorate every arrow with the subcategory it lives in. Now an object of our category is $T\in T$ together with an arrow in $A^{\eff}(C,\iota E,F)$ from $I$ and an arrow in $A^{\eff}(C, E, D)$ to $J$. These correspond to spans
    \begin{equation*}
        I \xleftarrow{F} T' \xrightarrow{\iota E} T
        \qquad\textrm{ and }\qquad
        T\xleftarrow{D} \tilde T \xrightarrow{E} J\,.
    \end{equation*}
    The last piece of data needed is an homotopy of their composition with the given map $I\to J$, that is a diagram
    \begin{equation*}
        \begin{tikzcd}
            & & J\\
            &\tilde J \ar{ur}{E}\ar{r}\ar{d}\ar{ld}{F} & \tilde T\ar{d}{D}\ar{u}{E}\\
            I&T' \ar{r}{\iota E} \ar{l}{F}& T\\
    \end{tikzcd}
    \end{equation*}
    where the central square is cartesian. But this is equivalent to the map $\tilde J\to \tilde T$ being an equivalence. Summing up, an object of our category is a factorization of $\tilde J\to I$
    \begin{equation*}
        \tilde J \xrightarrow{D} T\xrightarrow{F} I\,.
    \end{equation*}
    Moreover a similar analysis on higher simplices shows that this is the opposite of the category of factorizations. It is easy to see that $\tilde J \xrightarrow{=} \tilde J \to I$ is a terminal object for this category, which is then weakly contractible.
\end{proof}

\begin{thm}
    Let $C$ be a $T$-$\infty$-category with finite $T$-limits. Precomposition with the inclusion $j:\underline{\FF_*}_T\to \underline{\AA}^{\eff}(T)$ induces an equivalence
    \[\underline{\Fun}^\times_T(\underline{\AA}^{\eff}(T),C)\to \underline{\CMon}_T(C)\,.\]
    We denote the category on the left hand side by $\underline{\Mack}^T(C)$ and call it the \emph{$T$-$\infty$-category of $T$-Mackey functors valued in $C$}.
\end{thm}
\begin{proof}
    Let $\underline{\mathrm{Psh}}_T(C)$ be the $T$-presheaf $T$-$\infty$-category of $C$. Then, by the full faithfulness of the $T$-Yoneda embedding (see \cite[Th.~10.4]{Exp1}) we have a pullback square
    \begin{equation*}
    \begin{tikzpicture}
    \matrix(m)[matrix of math nodes,
    row sep=4ex, column sep=4ex,
    text height=1.5ex, text depth=0.25ex]{
    \underline{\Fun}^\times_T(\underline{\AA}^{\eff}(T),C)& \underline{\CMon}_T(C)\\
    \underline{\Fun}^\times_T(\underline{\AA}^{\eff}(T),\mathrm{Psh}_T(C))& \underline{\CMon}_T(\underline{\mathrm{Psh}}_T{C})\\
    \underline{\Fun}_T\left(C^{\vee\op},\underline{\Fun}^\times_T(\underline{\AA}^{\eff}(T),\underline{\Top}_T)\right) & \underline{\Fun}_T\left(C^{\vee\op},\underline{\CMon}_T(\underline{\Top}_T)\right)\\
    };
    \path[>=stealth,->,font=\scriptsize]
    (m-1-1) edge (m-1-2)
            edge (m-2-1)
    (m-1-2) edge (m-2-2)
    (m-2-1) edge (m-2-2)
            edge[-,double]  (m-3-1)
    (m-2-2) edge[-,double]  (m-3-2)
    (m-3-1) edge (m-3-2);
    \end{tikzpicture}\,,
    \end{equation*}
    where $C^{\vee\op}$ is the fiberwise opposite of \cite{BGN}, so it is enough to show the thesis for $C=\underline{\Top}_T$.
    
    We claim that sending every $T$-commutative monoid to its right Kan extension is the inverse of the restriction map. The first step is showing that the natural map
    \[j_*M\circ j \to M\]
    is an equivalence. By applying \ref{lem:Kan-extension-of-monoids} with $C=F=S_T$, $E$ the subcategory of fiberwise arrows and $D$ the category of summand inclusions (that is the egressive maps in the definition of $\underline{\FF_*}_T$) we see that it is enough to prove that the map
    \[k_*(M|_{\underline{\FF_*^{in}}_T})\circ k \to M|_{\underline{\FF_*^{in}}_T}\]
    is an equivalence, where $k$ is the inclusion of $E^\op$ in $F^\op$. But this follows immediately from the fact that $M|_{\underline{\FF_*^{in}}_T}$ is the right Kan extension of its restriction to $T^\op$.
    
    Hence $j_*M$ must be a product preserving functor (since the image of $j$ contains all product diagrams in $\underline{\AA}^{\eff}(T)$. Viceversa let suppose that $N$ is a product preserving functor from $\underline{\AA}^{\eff}(T)$ to $C$. Then there is a natural map
    \[N\to j_*(N\circ j)\,.\]
    But since $j$ is essentially surjective we can check that this is an equivalence after precomposing with $j$, which follows immediately from the previous case.
\end{proof}

\begin{nul}
    With a similar proof it is possible to prove that if $D$ is an $\infty$-category with finite products there is an equivalence
    \begin{equation*}
        \Fun^\times(A^{\eff}(T),D)\cong \CMon_T(D^T)\,.
    \end{equation*}
\end{nul}


    \section{$T$-linear functors and $T$-stability}
    
    Recall the definition of fiberwise linear functor and fiberwise stable cocartesian fibration from section \ref{sec:relative-stability}.
    
    The following definition is inspired to hypothesis (A) of \cite{MR2286026}.
    \begin{dfn}
        Let $C$ be a pointed $T$-$\infty$-category with finite $T$-colimits and let $D$ be a $T$-$\infty$-category with finite $T$-limits. Then a $T$-functor $F:\fromto{C}{D}$ is $T$-linear if it is fiberwise linear and $T$-semiadditive. A $T$-$\infty$-category with all finite $T$-limits and $T$-colimits is $T$-stable if it is fiberwise stable and $T$-semiadditive.
        
        We will denote the $T$-subcategory of $\underline{\Fun}_T(C,D)$ which on the fiber above $V$ is spanned by $T_{/V}$-linear functors from $C\times_{T^\op}(T_{/V})^\op$ to $D\times_{T^\op}(T_{/V})^\op$ with $\underline{\Lin}^T(C,D)$.
    \end{dfn}
    
    \begin{lem}
        Let $D$ be a $T$-semiadditive $T$-$\infty$-category. Then $\underline{\Sp}_T(D)$ is $T$-semiadditive (and hence $T$-stable) and the functor $\Omega^\infty:\underline{\Sp}_T(D)\to D$ preserves $T$-products (and so all $T$-limits).
    \end{lem}
    \begin{proof}
        Recall that $\underline{\Sp}_T(D)$ is the cocartesian fibration classified by the functor $V\mapsto \Sp(D_V)$ so it is clearly fiberwise semiadditive and we just need to show that for every arrow $W\to V$ in $T$ the pushforward functor
        \[\Sp(D_V)\to \Sp(D_W)\]
        has a coinciding left and right adjoints that satisfies the Beck-Chevalley condition. But since the left and right adjoint are clearly given by postcomposition of those of $D_V\to D_W$ the thesis follows.
    \end{proof}
    
    \begin{dfn}
        Let $D$ be a $T$-$\infty$-category with all finite $T$-limits. Then the $T$-$\infty$-category of $T$-spectra is
        \[\underline{\Sp}^T(D)=\underline{\Sp}_T(\underline{\CMon}_T(D))\,.\]
        By the previous lemma the latter category is $T$-stable
    \end{dfn}
    
    Note that there is a natural $T$-functor $\Omega^\infty:\fromto{\underline{\Sp}^T(D)}{D}$ given by the composition
    \[\underline{\Sp}^T(D)=\underline{\Sp}_T(\underline{\CMon}_T(D))\xrightarrow{\Omega^\infty}\underline{\CMon}_T(D)\xrightarrow{I(-)_+^*} D\,.\]
    It is immediate by the previous lemma that it preserves all $T$-limits.
    
    \begin{thm}[Universal property of $T$-spectra]\label{thm:universal-property-B-spectra}
        Let $C$ be a pointed $T$-$\infty$-category with finite $T$-colimits and $D$ be a $T$-$\infty$-category with finite $T$-limits. Then the functor
        \[(\Omega^\infty)_*\colon\fromto{\underline{\Fun}_T^{T-\mathrm{rex}}(C,\underline{\Sp}^T(D))}{\underline{\Lin}^T(C,D)}\]
        is an equivalence of $T$-$\infty$-categories, where the source categories is the full subcategory of those functors preserving finite $T$-limits. In particular
        \[\underline{\Sp}^T(D)\cong \underline{\Lin}_T(\underline{\Top_*}^{fin}_T,D)\,,\]
        and the functor $\Omega^\infty$ is given by evaluation at the cocartesian section $I(-)_+:T^\op\to \underline{\FF_*}_T$.
    \end{thm}
    \begin{proof}
        Since the map is clearly a $T$-functor we just need to check that it is an equivalence fiberwise. But, remembering that finite $T$-colimits are generated by $T$-coproducts and finite fiberwise colimits, we can apply \ref{cor:universal-property-monoids} and \ref{prp:fiberwise-spectra} and conclude
        \[\Fun^{T-\mathrm{rex}}_T(C,\underline{\Sp}_T(\underline{\CMon}_T(D))) = \Lin^\amalg_T(C,\CMon_T(D)) = \Lin^T(C,D)\,.\]
    \end{proof}

\begin{cor}
    Let $D$ be an $\infty$-category with all finite limits. Then there is an equivalence
    \[\Lin^T(\underline{\Top_*}^{fin}_T,\underline{D}_T) \cong \Fun^\oplus (A^{eff}(T),\Sp(D))\,.\]
\end{cor}
\begin{proof}
    Both of those are equivalent to the global sections of $\underline{\Sp}^T(\underline{D}_T)$.
\end{proof}


\appendix

\section{Comparison with orthogonal spectra}\label{sec:orthogonal}
    \begin{nul}
    Let $G$ be a finite group. In this appendix we will prove that our notion of $G$-spectra coincides with the orthogonal $G$-spectra developed in \cite{MR1922205}, thus reproving a theorem by Guillou and May (\cite{guillou2}). Fix once and for all a complete $G$-universe $\mathcal{U}$ (that is an isometric $G$-action on $\mathbb{R}^\infty$ such that every finite-dimensional representation can be isometrically embedded in $\mathbb{R}^\infty$ countably many times) and note that its restriction to a subgroup $H$ of $G$ is an $H$-universe, which we will take with indexing set given by the $G$-invariant subspaces. The category of orthogonal $H$-spectra with respect to $U$ (\cite[Df.~II.2.6]{MR1922205}) will be denoted by $\mathrm{Sp}^H_{(1)}$.
    \end{nul}
    
    \begin{dfn}
        Let $(\OO_G)_*$ be the category of $G$-orbits with a distinguished basepoint but with any possible map.\footnote{Another way of thinking of this category is as the category of $G$-orbits together with an explicit isomorphism with an orbit of the form $G/H$. This is of course purely bookkeeping and has nothing to do with the use of basepoints when defining $G$-spectra.},  We will write an element of $(\OO_G)_*$ as $G/H$ where the distinguished basepoint is $eH$. It is clear that the functor $(\OO_G)_*\to \OO_G$ that forgets the basepoint is an equivalence of categories. A map $G/H\to G/K$ is the datum of $gK\in G/K$ such that $g^{-1}Hg\subseteq K$. We have a functor from $(\OO_G)_*^\op$ to categories sending
    \begin{itemize}
        \item A pointed orbit $G/H$ to the category $\mathrm{Sp}_{(1)}^H$ of orthogonal $H$-spectra with respect to $U$;
        \item A map $G/H\to G/K$ the composition of the functors
        \[\mathrm{Sp}^K_{(1)}\to \mathrm{Sp}_{(1)}^{g^{-1}Hg} \to \mathrm{Sp}_{(1)}^H\]
        where the first functor is the restriction along the inclusion $g^{-1}Hg\subseteq K$ and the second functor is induced by the isomorphism $g^{-1}Hg\cong H$ given by conjugating by $g^{-1}$.
    \end{itemize}
    If we equip every category $\mathrm{Sp}_{(1)}^H$ with the family of $\pi_*$-isomorphisms (\cite[Df.~III.3.2]{MR1922205}) this becomes a functor from $(\OO_G)_*^\op$ to the category of relative categories (since \cite[Lm.~V.2.2]{MR1922205} implies that change of groups preserve $\pi_*$-isomorphisms). By precomposing with the equivalence $\OO_G^\op\cong (\OO_G)_*^\op$ and postcomposing with the localization functor from relative categories to $\infty$-categories we finally obtain a functor
    \[\OO_G^\op\to\Cat_\infty\]
    that classifies a cocartesian fibration $\underline{\Sp}^G_{orth}\to \OO_G^\op$. We call this cocartesian fibration the $G$-$\infty$-category of orthogonal $G$-spectra. It comes equipped with a natural $G$-functor
    \[\Omega^\infty:\underline{\Sp}^G_{orth}\to \underline{\Top}_G\]
    induced by the natural transformation obtained by sending every orthogonal $H$-spectrum to its 0-th space.
    \end{dfn}
    
    \begin{lem}
        The $G$-$\infty$-category $\underline{\Sp}^G_{orth}$ is $G$-stable
    \end{lem}
    \begin{proof}
        Since the fibers are obtained by localizing a stable model category at the weak equivalences $\Sp^G_{orth}$ is fiberwise stable. So we just need to check $G$-semiadditivity. But after unwrapping the definitions this is equivalent to the Wirthmüller isomorphism (\cite[Th.~II.6.2]{MR866482}, which holds for orthogonal $G$-spectra by \cite[Th.~III.4.16]{MR1922205}).
    \end{proof}
    
    We can now give a very simple proof of \cite[Th.~0.1]{guillou2} along the outline in section 11 of \cite{M2}.
    \begin{thm}[Guillou-May]
        The functor $\Omega^\infty:\Sp^G_{orth}\to \underline{\Top}_G$ lifts to an equivalence of $G$-$\infty$-categories $\underline{\Sp}^G_{orth}\cong \underline{\Sp}^G$.
    \end{thm}
    \begin{proof}
        Since the functor $\underline{\Sp}^G_{orth}\to \underline{\Top}_G$ preserves all finite $G$-limits (it has a left $G$-adjoint by proposition \cite[Pr.~7.3.2.1]{HA}) it lifts uniquely to a functor $\Xi:\Sp^G_{orth}\to \Sp^G$.
        
        For every orbit $V$ the fibers $(\Sp^G_{orth})_V$ and $(\Sp^G)_V$ are both generated by suspension spectra of orbits. Moreover $\Xi$ sends suspension spectra of orbits to suspension spectra of orbits and is fully faithful when restricted to those subcategories by \cite[Th.~10.6]{M2} and \cite[Th.~V.11.1]{MR1922205}, since in both settings $\Map(\Sigma^\infty_+G/H,\Sigma^\infty_+G/K)$ is just $\Omega^\infty\left(\Sigma^\infty_+(G/H\times G/K)\right)^G$. Hence it is an equivalence by the Schwede-Shipley theorem \cite[Th.~7.1.2.1]{HA}.
    \end{proof}

    From this description of $G$-spectra we immediately obtain a recognition principle for $G$-connective $G$-spectra
    \begin{cor}\label{cor:recognition-principle}
        There is an adjuction
        \[B\dashv\Omega^\infty\colon\CMon_G(\underline{\Top}_G)\leftrightarrows\Sp^G\]
        such that
        \begin{itemize}
            \item the unit $X\to\Omega^\infty B X$ is an equivalence if and only $X^H$ is a group-like monoid for every subgroup $H<G$;
            \item the counit $B\Omega^\infty E\to E$ is an equivalence if and only if $E^H$ is connective for every subgroup $H<G$.
        \end{itemize}
    \end{cor}
    \begin{proof}
        After our identifications this is just the adjunction
        \[\Fun^\times(A^{\eff}(G),\Top)\cong \Fun^\oplus(A^{\eff}(G),\CMon(\Top)) \leftrightarrows \Fun^\oplus(A^{\eff}(G),\Sp)\]
        given by postcomposition with the adjunction for ordinary spectra, and the thesis follows from the classical recognition theorem.
    \end{proof}


\bibliographystyle{amsplain}
\bibliography{Gcats}

\end{document}